\documentclass[12pt]{amsart}
\usepackage{latexsym,amssymb,amscd}

\def\B'c{{\mathcal{B'}}}
\def\U'c{{\mathcal{U'}}}

\def\opn#1#2{\def#1{\operatorname{#2}}} 
\opn\chara{char}
\opn\length{\ell}
\opn\projdim{proj\,dim}
\opn\injdim{inj\,dim}
\opn\ini{in}
\opn\rank{rank}
\opn\depth{depth}
\opn\height{ht}
\opn\embdim{emb\,dim}
\opn\codim{codim}

\opn\Tr{Tr}
\opn\bigrank{big\,rank}
\opn\superheight{superheight}\opn\lcm{lcm}
\opn\trdeg{tr\,deg}%
\opn\reg{reg}
\opn\lreg{lreg}
\opn\set{set}
\opn\supp{Supp}
\opn\shad{Shad}
\opn\del{del}
\opn\div{div}
\opn\Div{Div}
\opn\cl{cl}
\opn\Cl{Cl}
\opn\Spec{Spec}
\opn\Supp{Supp}
\opn\supp{supp}
\opn\Sing{Sing}
\opn\Ass{Ass}
\opn\Ann{Ann}
\opn\Rad{Rad}
\opn\Soc{Soc}
\opn\Ker{Ker}
\opn\Coker{Coker}
\opn\Im{Im}
\opn\Hom{Hom}
\opn\Tor{Tor}
\opn\Ext{Ext}
\opn\End{End}
\opn\Aut{Aut}
\opn\id{id}

\opn\nat{nat}
\opn\GL{GL}
\opn\SL{SL}
\opn\mod{mod}
\opn\ord{ord}
\opn\aff{aff}
\opn\con{conv}
\opn\relint{relint}
\opn\st{st}
\opn\lk{lk}
\opn\cn{cn}
\opn\core{core}
\opn\vol{vol}
\opn\gr{gr}

\def\pot#1#2{#1[\kern-0.28ex[#2]\kern-0.28ex]}

\opn\dirlim{\underrightarrow{\lim}}
\opn\invlim{\underleftarrow{\lim}}
\def\pnt{{\raise0.5mm\hbox{\large\bf.}}}

\def\twoline#1#2{\aoverb{\scriptstyle {#1}}{\scriptstyle {#2}}}
\newcommand{\aoverb}[2]{{\genfrac{}{}{0pt}{1}{#1}{#2}}}
\def\Implies{\ifmmode\Longrightarrow \else
     \unskip${}\Longrightarrow{}$\ignorespaces\fi}
\def\implies{\ifmmode\Rightarrow \else
     \unskip${}\Rightarrow{}$\ignorespaces\fi}
\def\iff{\ifmmode\Longleftrightarrow \else
     \unskip${}\Longleftrightarrow{}$\ignorespaces\fi}

\let\:=\colon
\newtheorem{Theorem}{Theorem}[section]
\newtheorem{Lemma}[Theorem]{Lemma}
\newtheorem{Corollary}[Theorem]{Corollary}
\newtheorem{Proposition}[Theorem]{Proposition}
\newtheorem{Remark}[Theorem]{Remark}

\newtheorem{Example}[Theorem]{Example}

\newtheorem{Definition}[Theorem]{Definition}

\let\epsilon=\varepsilon
\let\phi=\varphi
\let\kappa=\varkappa
\numberwithin{equation}{section}

\address{Faculty of Mathematics and Computer Science, Ovidius University, Bd.\ Mamaia 124,
 900527 Constanta, Romania,} \email{olteanuandageorgiana@gmail.com} 
\thanks{\footnotesize The author was supported by the grant CEX 05-D11-11/2005 and by the CNCSIS grant TD 507/2007.}

\title{A note on the subword complexes in Coxeter groups}
\author{Anda Olteanu}

\begin{document}

\begin{abstract} 
We prove that the Stanley--Reisner ideal of the Alexander dual of the subword complexes in Coxeter groups has linear quotients with respect to the lexicographical order of the minimal monomial generators. As a consequence, we obtain a shelling order on the facets of the subword complex. We relate some invariants of the subword complexes or of their dual with invariants of the word. For a particular class of subword complexes, we prove that the Stanley--Reisner ring is a complete intersection ring.
\end{abstract}
\maketitle

\section*{Introduction}
Subword complexes were introduced by A. Knutson and E. Miller \cite{KM} for the study of Schubert polynomials and combinatorics of the determinantal ideals. Many properties of the subword complexes in Coxeter groups are obtained by using the Demazure algebra and the Demazure product.

A. Knutson and E. Miller proved that subword complexes in Coxeter groups are vertex-decomposable \cite{KnMi}. Since any vertex-decomposable simplicial complex is shellable \cite{BiPr}, we have that subword complexes in Coxeter groups are shellable. 

In this paper we prove directly that subword complexes in Coxeter groups are shellable by using the Alexander duality. As a consequence, we get a shelling on the facets of the subword complex. We also study the Stanley--Reisner ideal of the Alexander dual for a particular class of subword complexes. For this class, we prove that the Stanley--Reisner ring is a complete intersection ring.

The paper is structured as follows: In the first section, we prove that the Stanley--Reisner ideal of the Alexander dual of the subword complexes in Coxeter groups has linear quotients with respect to the lexicographical order of its minimal monomial generators. As a consequence, we obtain a shelling on the facets of the subword complex. Next, we establish upper bounds for the projective dimension of the Stanley--Reisner ideal of the Alexander dual associated to a subword complex and for the regularity of the Stanley--Reisner ideal.

In the second section, we study a special class of subword complexes.

\section{Subword complexes in Coxeter groups and Alexander duality}

Let $(W,S)$ be an arbitrary Coxeter system consisting of a Coxeter group, $W$, and a set of simple reflections, $S$, that generates $W$ (see \cite{BjBr} or \cite{Hu} for background).

\begin{Definition}\cite{KM}\rm $\ $A \textit{word} $Q$ of size $n$ is an ordered sequence $Q=(\sigma_1,\ldots,\sigma_n)$ of elements of $S$. An ordered subsequence $P$ of $Q$ is called a \textit{subword} of $Q$.
\begin{enumerate}
	\item $P$ \textit{represents} $\pi\in W$ if the ordered product of the simple reflections in $P$ is a reduced expression for $\pi$.
	\item $P$ \textit{contains} $\pi\in W$ if some subsequence of $P$ represents $\pi$.
\end{enumerate}
 The \it{subword complex}\rm $\ \Delta(Q,\pi)$ is the set of all the subwords $Q\setminus P$ such that $P$ contains $\pi$.
\end{Definition}

\begin{Proposition} Let $\Delta$ be the subword complex $\Delta(Q,\pi)$ and let $n$ be the size of $Q$. Then $\projdim(k[\Delta])=\ell(\pi)$.
\end{Proposition} 

\begin{proof} By \cite[Theorem 2.5]{KnMi}, $\Delta$ is shellable. Since any subword $P\subseteq Q$ that represents $\pi$ is of size $\ell(\pi)$ and $\Delta$ is pure, we have that $\dim(\Delta)=n-\ell(\pi)-1$, so $\dim(k[\Delta])=n-\ell(\pi)$. Therefore, $\projdim(k[\Delta])=\ell(\pi)$.
\end{proof}

Let $Q=(\sigma_1,\ldots,\sigma_n)$ be a word in $W$, $\sigma_i\in S$ for all $1\leq i\leq n$ and $\pi$ an element in $W$. We consider the polynomial ring in $n$ variables over a field $k$, $k[x_1,\ldots,x_n]$, where $n$ is the size of the word $Q$. Let $\Delta(Q,\pi)$ be the subword complex. 
We aim to determine a shelling order on the facets of $\Delta(Q,\pi)$. For this purpose, we consider the Stanley--Reisner ideal of the Alexander dual associated to $\Delta(Q,\pi)$.

Let $\Delta$ be a simplicial complex on the vertex set $[n]$. We recall that
	\[I_{\Delta^{\vee}}=(\mathbf{x}_{[n]\setminus F}\ \mid\ F\in\mathcal{F}(\Delta)),
\]
where we denote by $\mathbf{x}_F$ the monomial $\prod\limits_{i\in F}x_i$ and by $\mathcal{F}(\Delta)$ the set of all the facets of $\Delta$. In the special context of the subword complexes, we obtain that
	\[I_{\Delta^{\vee}}=(\mathbf{x}_P\ \mid\ P\subseteq Q,\ P\ \mbox{represents}\ \pi).
\]

Let $R=k[x_1,\ldots,x_n]$ be the polynomial ring in $n$ variables over a field $k$ and $I$ a monomial ideal of $R$. $I$ is called an \textit{ideal with linear quotients} if there exists an order of the monomials from the minimal monomial generating system $u_1,\ldots,u_m$ such that, for all $i\geq2$ and for all $j<i$, there exist an integer $l\in[n]$ and an integer $1\leq k<i$ such that $u_k/[u_i,u_k]=x_l$ and $x_l$ divides $u_j/[u_i,u_j]$, where, for two monomials $u$ and $v$, we denote $[u,v]:=\gcd(u,v)$. 

Since $\Delta=\Delta(Q,\pi)$ is shellable (\cite[Theorem 2.5]{KnMi}) the Stanley--Reisner ideal of the Alexander dual of $\Delta$, $I_{\Delta^{\vee}}$, has linear quotients \cite[Theorem 1.4]{HeHiZh}. In order to obtain a shelling on the facets of $\Delta$, we have to define an order on the monomials from the minimal monomial system of generators of $I_{\Delta^{\vee}}$ such that $I_{\Delta^{\vee}}$ has linear quotients with respect to this order.

Firstly, let us fix some notations. Henceforth, we will write $Q\setminus \sigma_i$ for the word of size $n-1$ obtained from $Q$ by omitting $\sigma_i$, that is $Q\setminus \sigma_i=(\sigma_1,\ldots,\sigma_{i-1},\sigma_{i+1},\ldots,\sigma_n).$ Also, we will denote by $"\succ"$ the Bruhat order on $W$. 

For $P=(\sigma_{i_1},\ldots, \sigma_{i_m})$, $m\leq n$, a subword of $Q$, we will denote by $\mathbf{x}_{P}$ the monomial $x_{i_1}\ldots x_{i_{m}}$ in $k[x_1,\ldots,x_n]$ and by $\delta(P)$ the Demazure product of the word $P$ (see \cite{KnMi}). We use $G(I)$ for the minimal monomial generating set of the monomial ideal $I$. 

Now we can state the main result of this section:

\begin{Theorem}\label{linquot} Let $\Delta$ be the subword complex $\Delta(Q,\pi)$. Then the Stanley--Reisner ideal of the Alexander dual, $I_{\Delta^{\vee}}$, has linear quotients with respect to the lexicographi- cal order of the minimal monomial generators.
\end{Theorem}

\begin{proof} Let $G(I_{\Delta^{\vee}})=\{\mathbf{x}_{P_1},\ldots,\mathbf{x}_{P_{r}}\}$. We assume that $\mathbf{x}_{P_1}>_{lex}\ldots>_{lex}\mathbf{x}_{P_{r}}$. Note that $P_i\subseteq Q$ and $P_i$ represents $\pi$ for all $1\leq i\leq r$.

 We have to prove that $I_{\Delta^{\vee}}$ has linear quotients with respect to the sequence of monomials $\mathbf{x}_{P_1},\ldots,\mathbf{x}_{P_r}$, that is, for all $i\geq2$ and for all $j<i$, there exist an integer $l\in[n]$ and an integer $1\leq k<i$ such that $\mathbf{x}_{P_k}/[\mathbf{x}_{P_i},\mathbf{x}_{P_k}]=x_l$ and $x_l$ divides $\mathbf{x}_{P_j}/[\mathbf{x}_{P_i},\mathbf{x}_{P_j}]$. 

Let us fix $i\geq 2$ and $j<i$. Since $P_i,\ P_j$ represent $\pi$, they are subwords of $Q$ of size $\ell(\pi)$. Let $P_i=(\sigma_{i_1},\ldots,\sigma_{i_{\ell(\pi)}})$, $P_j=(\sigma_{j_1},\ldots,\sigma_{j_{\ell(\pi)}})$ and let $l\in [n]$ be an integer such that $i_t=j_t$ for all $1\leq t\leq l-1$ and $j_l<i_l$. Such an integer exists since $j<i$, that is $\mathbf{x}_{P_j}>_{lex}\mathbf{x}_{P_i}$.

Let $T$ be the subword of $Q$ of size $|T|=\ell(\pi)+1$ obtained from $P_i$ by adding $\sigma_{j_l}$ between $\sigma_{i_{l-1}}$ and $\sigma_{i_l}$, that is $$T=(\sigma_{i_1},\ldots,\sigma_{i_{l-1}},\sigma_{j_l},\sigma_{i_l},\ldots,\sigma_{i_{\ell(\pi)}}).$$ Since $T$ contains $\pi$, we have that $\delta(T)\succeq\pi$ (\cite[Lemma 3.4(1)]{KnMi}).

Let us assume that $\delta(T)\succ\pi$. Hence, $T$ represents an element $\tau\in W$, $\tau\succ\pi$ such that $\ell(\tau)=\ell(\pi)+1$ \cite[Lemma 3.4(3)]{KnMi}. Since $P_i,\ P_j$ represent $\pi$, we have that
	\[\pi=\sigma_{i_1} \ldots \sigma_{i_{l-1}}\sigma_{i_l}\ldots \sigma_{i_{\ell(\pi)}}=\sigma_{j_1} \ldots \sigma_{j_{l-1}}\sigma_{j_l}\ldots \sigma_{j_{\ell(\pi)}}
\]
are reduced expressions for $\pi$ and, by the choice of $l$, we obtain
\begin{eqnarray}\sigma_{j_l}\sigma_{i_l}\ldots \sigma_{i_{\ell(\pi)}}=\sigma_{j_{l+1}}\ldots \sigma_{j_{\ell(\pi)}}.\label{1}
\end{eqnarray}
Now, since $T$ represents $\tau$,
	\[\tau=\sigma_{i_1}\ldots \sigma_{i_{l-1}} \sigma_{j_l}\sigma_{i_l}\ldots \sigma_{i_{\ell(\pi)}}
\]
is a reduced expression for $\tau$. On the other hand, using the equality (\ref{1}), we have that $$\tau=\sigma_{i_1} \ldots \sigma_{i_{l-1}}\sigma_{j_{l+1}}\ldots \sigma_{j_{\ell(\pi)}}.$$ This is impossible since we obtained that $\tau$ can be written as a product of $\ell(\pi)-1$ simple reflections which contradicts the fact that $\ell(\tau)=\ell(\pi)+1$.

Hence, $\delta(T)=\pi$ and there exists a unique $\sigma_{i_t}\neq \sigma_{j_l}$ such that $T\setminus \sigma_{j_l}=P_i$ and $T\setminus \sigma_{i_t}$ represents $\pi$ \cite[Lemma 3.5(2)]{KnMi}. Let us denote $P'=T\setminus \sigma_{i_t}$. We will show that $\mathbf{x}_{P'}>_{lex}\mathbf{x}_{P_i}$ which will end the proof. One may note that $\mathbf{x}_{P\,'}=x_{j_l}\mathbf{x}_P/x_{i_t}$ and $\mathbf{x}_{P\,'}\neq \mathbf{x}_P$ since $j_l\neq i_t$.

Assume by contradiction that $\mathbf{x}_{P'}<_{lex}\mathbf{x}_{P_i}$, that is $i_t<j_l$. Since both $P_i$ and $P\,'$ represent $\pi$, we have that
	\[\pi=\sigma_{i_1}\ldots \sigma_{i_t}\ldots \sigma_{i_{l-1}}\sigma_{i_l}\ldots \sigma_{i_{\ell(\pi)}}=\sigma_{i_1}\ldots \sigma_{i_{t-1}}\sigma_{i_{t+1}}\ldots \sigma_{i_{l-1}}\sigma_{j_l}\sigma_{i_l}\ldots \sigma_{i_{\ell(\pi)}}
\]
are two reduced expressions for $\pi$. The above equality can be written as
	\begin{eqnarray}\sigma_{i_t}\sigma_{i_{t+1}}\ldots \sigma_{i_{l-1}}=\sigma_{i_{t+1}}\ldots \sigma_{i_{l-1}}\sigma_{j_l}.\label{x}
\end{eqnarray}
On the other hand, $P_j$ also represents $\pi$, that is $\pi=\sigma_{j_1}\ldots \sigma_{j_l}\ldots \sigma_{j_{\ell(\pi)}}$. Now, since for all $1\leq k<l$, $i_k=j_k$, using (\ref{x}) we obtain that $$\pi=\sigma_{i_1}\ldots \sigma_{i_{t-1}}\sigma_{i_{t+1}}\ldots \sigma_{i_{l-1}}\sigma_{j_{l+1}}\ldots \sigma_{j_{\ell(\pi)}}.$$ Hence, we get an expression for $\pi$ with $\ell(\pi)-2$ simple reflections, which is impossible. Thus, we must have $\mathbf{x}_{P'}>_{lex}\mathbf{x}_{P_i}$. So there exist a monomial $\mathbf{x}_{P\,'}$ and an integer $j_l\in[n]$ such that $\mathbf{x}_{P\,'}/[\mathbf{x}_{P\,'},\mathbf{x}_{P_i}]=x_{j_l}$ and $x_{j_l}$ divides $\mathbf{x}_{P_{j}}/[\mathbf{x}_{P_i},\mathbf{x}_{P_j}]$ which ends our proof.
\end{proof}

\begin{Example}\label{ex}\rm Let $(S_4,S)$ be the Coxeter system, where $S_4$ is the symmetric group and $S$ is the set of the adjanced transpositions, that is $S=\{s_1=(1,2),\ s_2=(2,3),\ s_3=(3,4)\}$. Let $Q$ be the word of size $8$, $Q=(s_1,\ s_2,\ s_1,\ s_3,\ s_1,\ s_2,\ s_3,\ s_1)$ and $\pi=(1,2,4)\in S_4$ with $\ell(\pi)=4$. The set of all the reduced expression of $\pi$ is $\{s_1s_2s_3s_2,\ s_1s_3s_2s_3,\ s_3s_1s_2s_3\}.$ Let us denote $Q=(\sigma_1,\ \sigma_2,\ \sigma_3,\ \sigma_4,\ \sigma_5,\ \sigma_6,\ \sigma_7,\ \sigma_8).$ The set of all the subwords of $Q$ that represent $\pi$ is $\{(\sigma_1,\sigma_2,\sigma_4,\sigma_6),\ (\sigma_1,\sigma_4,\sigma_6,\sigma_7), $ $(\sigma_3,\sigma_4,\sigma_6,\sigma_7),\ (\sigma_4,\sigma_5,\sigma_6,\sigma_7)\}.$ Hence, the subword complex $\Delta=\Delta(Q,\pi)$ is the simplicial complex with the facets  $\{\sigma_3,\sigma_5,\sigma_7,\sigma_8\},\ \{\sigma_2,\sigma_3,\sigma_5,\sigma_8\},\ \{\sigma_1,\sigma_2,\sigma_5,\sigma_8\},$ $\{\sigma_1,\sigma_2,\sigma_3,\sigma_8\}$.

Let $k[x_1,\ldots,x_8]$ be the polynomial ring over a field $k$. The Stanley--Reisner ideal of the Alexander dual of $\Delta$ is the square-free monomial ideal whose minimal monomial system of generators is $G(I_{\Delta^{\vee}})=\{x_1x_2x_4x_6,\ x_1x_4x_6x_7,\ x_3x_4x_6x_7,\ x_4x_5x_6x_7\}.$ We denote $\mathbf{x}_{P_1}=x_1x_2x_4x_6,\ \mathbf{x}_{P_2}=x_1x_4x_6x_7,\ \mathbf{x}_{P_3}=x_3x_4x_6x_7,\ \mathbf{x}_{P_4}=x_4x_5x_6x_7$. Then we have $\mathbf{x}_{P_1}>_{lex}\ldots>_{lex}\mathbf{x}_{P_4}$. Since $(\mathbf{x}_{P_1})\colon \mathbf{x}_{P_2}=(x_2),\ (\mathbf{x}_{P_1},\mathbf{x}_{P_2})\colon \mathbf{x}_{P_3}=(x_1)$ and $(\mathbf{x}_{P_1},\mathbf{x}_{P_2},\mathbf{x}_{P_3})\colon \mathbf{x}_{P_4}=(x_1,x_3)$, $I_{\Delta^{\vee}}$ has linear quotients with respect to this order of the monomials from $G(I_{\Delta^{\vee}})$. 
\end{Example}

We can state the following corollary:

\begin{Corollary}\label{shell} Let $\Delta$ be the subword complex $\Delta(Q,\pi)$ and let $G(I_{\Delta^{\vee}})=\{\mathbf{x}_{P_1},\ldots,$ $\mathbf{x}_{P_r}\}$, where $\mathbf{x}_{P_1}>_{lex}\ldots>_{lex}\mathbf{x}_{P_r}$. Then $Q\setminus P_{1},\ldots, Q\setminus P_r$ is a shelling order on the facets of $\Delta$.
\end{Corollary}
\begin{proof} Since $G(I_{\Delta^{\vee}})=\{\mathbf{x}_{P_1},\ldots,\mathbf{x}_{P_r}\}$, we have that $Q\setminus P_1,\ldots,Q\setminus P_r$ are the facets of $\Delta$. Since $(Q\setminus P_i)\setminus(Q\setminus P_j)=P_j\setminus P_i$ the assertion follows from the characterization of a square-free monomial ideal with linear quotients \cite{H}. 
\end{proof}

\begin{Remark}\rm The shelling from Corollary \ref{shell} for the subword complex $\Delta(Q,\pi)$ coincides with the shelling inductively constructed by vertex-decomposing the subword complex $\Delta(Q,\pi)$.
\end{Remark}

\begin{Example}\rm We study the same subword complex as in Example \ref{ex}. We have that $F_1=\{\sigma_3,\sigma_5,\sigma_7,\sigma_8\},\ F_2=\{\sigma_2,\sigma_3,\sigma_5,\sigma_8\},\ F_3=\{\sigma_1,\sigma_2,\sigma_5,\sigma_8\},\ F_4=\{\sigma_1,\sigma_2,\sigma_3,\sigma_8\}$, is a shelling on the facets of $\Delta$. We shall prove that the same shelling is obtained inductively by vertex-decomposing $\Delta$. 

Let $Q'=Q\setminus \sigma_1$. Since $\ell(\sigma_1\pi)<\ell(\pi)$, by the proof of \cite[Theorem 2.5]{KnMi}, one obtains that $$\lk(\sigma_1,\Delta)=\Delta(Q',\pi)=\langle\{\sigma_2,\sigma_5,\sigma_8\},\ \{\sigma_2,\sigma_3,\sigma_8\}\rangle$$ and $$\del(\sigma_1,\Delta)=\del(Q',\sigma_1\pi)=\langle\{\sigma_3,\sigma_5,\sigma_7,\sigma_8\},\ \{\sigma_2,\sigma_3,\sigma_5,\sigma_8\}\rangle.$$ We denote $\Delta_1=\langle\{\sigma_2,\sigma_5,\sigma_8\},\ \{\sigma_2,\sigma_3,\sigma_8\}\rangle$ and $\Delta_2=\langle\{\sigma_3,\sigma_5,\sigma_7,\sigma_8\},\ \{\sigma_2,\sigma_3,\sigma_5,\sigma_8\}\rangle$. We apply the same procedure to $\Delta_1$. Let $Q''=Q'\setminus\sigma_2$. Since $\ell(\sigma_2\pi)>\ell(\pi)$, we have $$\lk(\sigma_2,\Delta_1)=\del(\sigma_2,\Delta_1)=\Delta(Q'',\pi)=\langle\{\sigma_5,\sigma_8\},\ \{\sigma_3,\sigma_8\}\rangle.$$ 
Let us denote this simplicial complex by $\Delta_1'$ and $Q'''=Q''\setminus \sigma_2$. We have $\ell(\sigma_3\pi)<\ell(\pi)$. Hence, $$\lk(\sigma_3,\Delta_1')=\Delta(Q''',\pi)=\langle\{\sigma_8\}\rangle$$ and $$\del(\sigma_3,\Delta_1')=\Delta(Q''',\sigma_3\pi)=\Delta(Q''',s_2s_3s_2)=\langle\{\sigma_5,\sigma_8\}\rangle.$$

For the simplicial complex $\Delta_2$, since $\ell(\sigma_2\sigma_1\pi)<\ell(\sigma_1\pi)$, one has that $$\lk(\sigma_2,\Delta_2)=\Delta(Q'',\sigma_1\pi)=\Delta(Q'',s_2s_3s_2)=\langle\{\sigma_3,\sigma_5,\sigma_8\}\rangle$$ and $$\del(\sigma_2,\Delta_2)=\Delta(Q'',\sigma_2\sigma_1\pi)=\Delta(Q'',s_3s_2)=\langle\{\sigma_3,\sigma_5,\sigma_7,\sigma_8\}\rangle.$$

Hence, we get the following shelling on the facets of the subword complex $\Delta$
	\[\{\sigma_3,\sigma_5,\sigma_7,\sigma_8\},\ \{\sigma_2,\sigma_3,\sigma_5,\sigma_8\},\ \{\sigma_1,\sigma_2,\sigma_5,\sigma_8\},\ \{\sigma_1,\sigma_2,\sigma_3,\sigma_8\}
\]
which is the same shelling as the one obtained in the Example \ref{ex}. 
\end{Example}

Recall that a pure simplicial complex $\Delta$ with the vertex set $\{v_1,\ldots, v_n\}$ is \textit{shifted} if there exists a labelling of the vertices by $1$ to $n$ such that, for any face $F\in \Delta$, replacing any vertex $v_i\in F$
by a vertex with a smaller label that does not belong to $F$, we get a set which is also a face of $\Delta$.

Note that subword complexes are not shifted simplicial complexes, as can be seen in the following example.

\begin{Example}\rm Let $(S_4,S)$ be the Coxeter system, $Q$ the word $(s_1,s_3,s_3,s_1,s_2,s_3),$ and $\pi=(1,2,4)\in S_4$, $\ell(\pi)=4$. We denote $Q=(\sigma_1,\ \sigma_2,\ \sigma_3,\ \sigma_4,\ \sigma_5,\ \sigma_6).$ The set of all the subwords of $Q$ that represent $\pi$ is $\{(\sigma_1,\sigma_2,\sigma_5,\sigma_6),\ (\sigma_1,\sigma_3,\sigma_5,\sigma_6),\ (\sigma_2,\sigma_4,\sigma_5,$ $\sigma_6),\ (\sigma_3,\sigma_4,\sigma_5,\sigma_6)\}.$ Hence, the subword complex $\Delta=\Delta(Q,\pi)$ is the simplicial complex with the facets set $\mathcal{F}(\Delta)=\{\{\sigma_3,\sigma_4\},\ \{\sigma_2,\sigma_4\},\ \{\sigma_1,\sigma_3\},\ \{\sigma_1,\sigma_2\}\}$. If we consider a label of the vertices $\sigma_1,\ldots,\sigma_4$, such that $\sigma_1<\sigma_3$, looking at the facet $\{\sigma_3,\sigma_4\}$ and replacing $\sigma_3$ by $\sigma_1$, we obtain the set $\{\sigma_1,\sigma_4\}$ which is not a face in $\Delta$. If we order the vertices such that $\sigma_3<\sigma_1$, looking at the facet $\{\sigma_1,\sigma_2\}$ and replacing $\sigma_1$ by $\sigma_3$, we obtain the set $\{\sigma_2,\sigma_3\}$ which is not a face in $\Delta$. Hence, $\Delta$ is not a shifted simplicial complex.
\end{Example}

Let $I$ be a monomial ideal of $k[x_1,\ldots,x_n]$ with $G(I)=\{w_1,\ldots,w_r\}$. Assume that $I$ has linear quotients with respect to the sequence $w_1,\ldots,w_r$. We denote $\set(w_i)=\{k\in[n]\ :\ x_{k}\in(w_1,\ldots,w_{i-1})\colon w_i\},$ for all $2\leq i\leq r$.

Next we determine all the elements of $\set(\mathbf{x}_P)$, where $\mathbf{x}_P\in G(I_{\Delta^{\vee}})$ and the monomials from $G(I_{\Delta^{\vee}})$ are ordered decreasing in the lexicographical order. For this purpose, let us fix some notations. For a monomial $w$ in $k[x_1,\ldots,x_n]$, we will denote $\max(w)=\max\{i\in[n]\ :\  x_i|w\}$ and $\min(w)=\min\{i\in[n]\ :\ x_i|w\}.$
 
 \begin{Lemma}\label{snmid} Let $I$ be a square-free monomial ideal with $G(I)=\{w_1,\ldots,w_r\}$ and $w_1>_{lex}\ldots>_{lex}w_r$ such that $I$ has linear quotients with respect to this order of the generators. Then
	\[\set(w_i)\subseteq [\max(w_i)]\setminus \supp(w_i),
\]
where $[\max(w_i)]=\{1,2,\ldots,\max(w_i)\}$ and $\supp(w_i)=\{k\in[n]\ :\ x_k\mid w_i\}$.
 \end{Lemma}
 \begin{proof} Let $i\geq2$ and $k\in\set(w_i)$. Then $x_kw_i\in(w_1,\ldots,w_{i-1})$, that is there exist a variable $x_t$, $t\neq k$ and $j<i$ such that $w_ix_k=w_jx_t$. If $k\in\supp(w_i)$, since $k\neq t$, we get that $x_k^2\mid w_j$, contradiction. Thus $k\notin\supp(w_i)$. Since $t\neq k$, we have $x_t\mid w_i$. Then $w_j=x_kw_i/x_t>_{lex} w_i$,
which implies $k<t\leq\max(w_i)$.
 \end{proof}
 In general, the inclusion is strict, as one can see in the following example.
\begin{Example}\rm We consider the same subword complex as in Example \ref{ex}. We note that $\set(\mathbf{x}_{P_2})=\{2\}$, $\set(\mathbf{x}_{P_3})=\{1\}$ and $\set(\mathbf{x}_{P_4})=\{1,3\}$. We have $[\max(\mathbf{x}_{P_4})]\setminus\supp\{\mathbf{x}_{P_4}\}=\{1,2,3\},$ and $\set(\mathbf{x}_{P_4})\subsetneq\{1,2,3\}$.
\end{Example}

We will denote by $P_j\setminus P_i$ the word obtained from $P_j$ by omitting the simple reflections that appear both in $P_i$ and $P_j$. We will use $\min(P_j\setminus P_i)$ for $\min(\mathbf{x}_{P_j}/[\mathbf{x}_{P_j},\mathbf{x}_{P_i}])$.

\begin{Proposition}\label{set} Let $\Delta$ be the subword complex $\Delta(Q,\pi)$ and $G(I_{\Delta^{\vee}})=\{\mathbf{x}_{P_1},\ldots,$ $\mathbf{x}_{P_r}\}$, where $\mathbf{x}_{P_1}>_{lex}\ldots>_{lex}\mathbf{x}_{P_r}$. For any $1\leq i\leq r$, we have $$\set(\mathbf{x}_{P_i})=\{\min(P_j\setminus P_i)\ \colon\ 1\leq j<i\}.$$
\end{Proposition}

\begin{proof} We will use $I$ instead of $I_{\Delta^{\vee}} $ to simplify the notation. 

Let $s\in\set(\mathbf{x}_{P_i})$. Since $s\in\set(\mathbf{x}_{P_i})$, $x_s\mathbf{x}_{P_i}\in (\mathbf{x}_{P_1},\ldots,\mathbf{x}_{P_{i-1}})$. Hence, there exists $j<i$ and a variable $x_t$ such that 
	$x_s\mathbf{x}_{P_i}=x_t\mathbf{x}_{P_j}.$
One may note that $s\neq t$ since $\mathbf{x}_{P_i}\neq\mathbf{x}_{P_j}$. By Lemma \ref{snmid}, we have that $x_s\nmid\mathbf{x}_{P_i}$. Thus, $\mathbf{x}_{P_j}=x_s\mathbf{x}_{P_i}/x_t$ and the assertion follows.

The other inclusion follows from the proof of Theorem \ref{linquot}. 
\end{proof}

We note that, for an arbitrary square-free monomial ideal which has linear quotients with respect to the lexicographical order of its minimal monomial generators $w_1,\ldots,w_r$, the equality $\set(w_i)=\{\min(\supp(w_j)\setminus\supp(w_i))\ :\ 1\leq j<i\}$ might be not true.

\begin{Example}\rm Let $I=(x_1x_2x_3,\ x_2x_3x_4,\ x_2x_4x_5)$ be a square-free monomial ideal in the polynomial ring $k[x_1,\ldots,x_5]$. Denote $w_1=x_1x_2x_3,\ w_2= x_2x_3x_4,\ w_3=x_2x_4x_5$. One may note that $w_1>_{lex}w_2>_{lex}w_3$ and $I$ has linear quotients with respect to this order of the generators. We have that $\set(w_2)=\{1\}$ and $\set(w_3)=\{3\}$. If we denote $F_i=\supp(w_i)$, $1\leq i\leq3$, we have that $\min(F_1\setminus F_3)=\{1\}$ and $\{1\}\notin\set(w_3)$.
\end{Example}
 
Let $\Delta$ be the subword complex $\Delta(Q,\pi)$ and $G(I_{\Delta^{\vee}})=\{\mathbf{x}_{P_1},\ldots,\mathbf{x}_{P_r}\}$ be the minimal monomial system of generators for $I_{\Delta^{\vee}}$ with $\mathbf{x}_{P_1}>_{lex}\ldots>_{lex}\mathbf{x}_{P_r}$. For a monomial $\mathbf{x}_{P_i}$ from $G(I_{\Delta^{\vee}})$, we denote $d_i=|\set(\mathbf{x}_{P_i})|$. We note that, by Proposition \ref{set}, we have $d_i\leq i-1$. One may easily find examples with $d_i=i-1$ for all $i$.

\begin{Theorem}\label{pd} Let $Q=(\sigma_1,\ldots,\sigma_n)$ be a word in $W$, $\pi$ an element in $W$ and $\Delta$ the subword complex $\Delta(Q,\pi)$. Then $\projdim(I_{\Delta^{\vee}})\leq n-\ell(\pi).$
\end{Theorem}
\begin{proof} Let $G(I_{\Delta^{\vee}})=\{\mathbf{x}_{P_1},\ldots,\mathbf{x}_{P_r}\}$, where $\mathbf{x}_{P_1}>_{lex}\ldots>_{lex}\mathbf{x}_{P_r}$. Since $I_{\Delta^{\vee}}$ has linear quotients with respect to the sequence $\mathbf{x}_{P_1},\ldots,\mathbf{x}_{P_r}$, we have $\projdim(I_{\Delta^{\vee}})=\max\{d_1,\ldots,d_r\}$, \cite{H}. Let us assume by contradiction that $\projdim(I_{\Delta^{\vee}})$ $> n-\ell(\pi)$. Hence, there exists $1\leq k\leq r$ such that $\projdim(I_{\Delta^{\vee}})=d_k> n-\ell(\pi).$ By Lemma \ref{snmid}, we have $\set(\mathbf{x}_{P_k})\cap\supp(\mathbf{x}_{P_k})=\emptyset.$ Since $\mathbf{x}_{P_k}$ is a square-free monomial, $|\supp(\mathbf{x}_{P_k})|=\ell(\pi)$. We have that 
	\[|\set(\mathbf{x}_{P_k})|+|\supp(\mathbf{x}_{P_k})|>n-\ell(\pi)+\ell(\pi)=n
\]
that is $|\set(\mathbf{x}_{P_k})\cup\supp(\mathbf{x}_{P_k})|>n$, which is a contradiction.
\end{proof}

\begin{Remark}\rm Let $Q=(\sigma_1,\ldots,\sigma_n)$ be a word in $W$, $\pi$ an element in $W$ and $\Delta$ the subword complex $\Delta(Q,\pi)$. Let $G(I_{\Delta^{\vee}})=\{\mathbf{x}_{P_1},\ldots,\mathbf{x}_{P_r}\}$ be the minimal monomial system of generators for $I_{\Delta^{\vee}}$ with $\mathbf{x}_{P_1}>_{lex}\ldots>_{lex}\mathbf{x}_{P_r}$. By Theorem \ref{pd}, we have that if there exists $1\leq i\leq r$ such that $d_i=i-1$ then $i\leq n-\ell(\pi)+1$.
\end{Remark}

\begin{Corollary}Let $Q=(\sigma_1,\ldots,\sigma_n)$ be a word in $W$, $\pi\in W$ be an element and $\Delta$ the subword complex $\Delta(Q,\pi)$. Then $\reg(I_{\Delta})\leq n-\ell(\pi)+1.$
\end{Corollary}
\begin{proof} By Terai's theorem, \cite{T}, we have that $\reg(I_{\Delta})=\projdim k[\Delta^{\vee}].$ On the other hand, by Theorem \ref{pd}, $\projdim k[\Delta^{\vee}]=\projdim(I_{\Delta^{\vee}})+1\leq n-\ell(\pi)+1$
which ends our proof.
\end{proof}

The following results will be used in the last section of this paper.

\begin{Lemma}\label{min} Let $u,v,w$ be monomials of the same degree in $k[x_1,\ldots,x_n]$. Assume that $u,v>_{lex}w$ and $\min(u/\gcd(u,w))\neq \min(v/\gcd(v,w))$. Then $\min(u/\gcd(u,w))$ $<\min(v/\gcd(v,w))$ if and only if $u>_{lex} v$.
\end{Lemma}
\begin{proof} In the following, for a monomial $m=x_1^{\alpha_1}\ldots x_n^{\alpha_n}$, we denote by $\nu_i(m)$ the exponent of the variable $x_i$ in $m$, that is $\nu_i(m)=\alpha_i$, $i=1,\ldots,n$. Since $u>_{lex}w$ there exists an integer $l'$ such that for all $i<l'$, $\nu_i(u)=\nu_i(w)$ and $\nu_{l'}(u)>\nu_{l'}(w)$. Similar, since $v>_{lex}w$ there exists an integer $l''$ such that for all $i<l''$, $\nu_i(v)=\nu_i(w)$ and $\nu_{l''}(v)>\nu_{l''}(w)$. By the hypothesis, we have $l'\neq l''$.

"$\Rightarrow$" The statement is obvious.

"$\Leftarrow$" Since $u>_{lex}v$, there exists an integer $l\in [n]$ such that for all $i<l$, $\nu_i(u)=\nu_i(v)$ and $\nu_{l}(u)>\nu_{l}(v)$.  One may easily check that the case $l\neq\min(l',l'')$ is impossible. Hence we must have $l=\min(l',l'')$. Let us assume that $l''<l'$. Hence $l=l''$ and we get that $\nu_{l''}(u)=\nu_{l''}(w)<\nu_{l''}(v)<\nu_{l''}(u)$ which is impossible. Thus, we must have $l'<l''$, that is $\min(u/\gcd(u,w))<\min(v/\gcd(v,w))$.
\end{proof}

\begin{Lemma}\label{i-1} Let $\Delta$ be the subword complex $\Delta(Q,\pi)$ and let $G(I_{\Delta^{\vee}})=\{\mathbf{x}_{P_1},\ldots,$ $\mathbf{x}_{P_r}\}$ be the minimal monomial system of generators for $I_{\Delta^{\vee}}$, with $\mathbf{x}_{P_1}>_{lex}\ldots>_{lex}$ $\mathbf{x}_{P_r}$. Assume that there exists $2\leq i\leq r$ such that $d_i=i-1$. Then, for all $1\leq j<i$, $d_j=j-1$. 
\end{Lemma}
\begin{proof} Since $d_i=i-1$, we have that $\min(P_j\setminus P_i)\neq\min(P_k\setminus P_i)$, for all $1\leq j,k<i$, $j\neq k$. Hence, by Lemma \ref{min}, $\min(P_1\setminus P_i)<\ldots<\min(P_{i-1}\setminus P_i)$. Let us fix $j<i$ and assume that $P_j=(\sigma_{j_1},\ldots,\sigma_{j_{\ell(\pi)}})$ and $P_i=(\sigma_{i_1},\ldots,\sigma_{i_{\ell(\pi)}})$. We have that, for all $t<\min(P_j\setminus P_i),\ i_t=j_t$ and $j_{\min(P_j\setminus P_i)}<i_{\min(P_j\setminus P_i)}.$

Let $1\leq k<j$. We prove that $\min(P_k\setminus P_i)=\min(P_k\setminus P_j)$. This will imply that, for all $1\leq k,s<j$ with $k\neq s$, $\min(P_k\setminus P_j)\neq\min(P_s\setminus P_j)$ and hence $d_j=j-1$.

Since $k<j$ and $\min(P_j\setminus P_i)\neq\min(P_k\setminus P_i)$, by Lemma \ref{min}, $\min(P_k\setminus P_i)<\min(P_j\setminus P_i).$ On the other hand, since $k<i$, we have that for all $t<\min(P_k\setminus P_i)$ $i_t=k_t$ and $k_{\min(P_k\setminus P_i)}<i_{\min(P_k\setminus P_i)}=j_{\min(P_k\setminus P_i)}$. 

We proved that for all $t<\min(P_k\setminus P_i)$ $k_t=j_t$ and $k_{\min(P_k\setminus P_i)}<j_{\min(P_k\setminus P_i)}$ which means that $\min(P_k\setminus P_i)=\min(P_k\setminus P_j)$.
\end{proof}

\begin{Lemma}\label{unic l} Let $\Delta$ be the subword complex $\Delta(Q,\pi)$ and let $G(I_{\Delta^{\vee}})=\{\mathbf{x}_{P_1},\ldots,$ $\mathbf{x}_{P_r}\}$ be the minimal monomial system of generators for $I_{\Delta^{\vee}}$, with $\mathbf{x}_{P_1}>_{lex}\ldots>_{lex}\mathbf{x}_{P_r}$. Let $2\leq i\leq r$. Then $d_i=i-1$ if and only if there exists a unique $l\in\supp(\mathbf{x}_{P_i})$ such that $\mathbf{x}_{P_j}=x_{\min(P_j\setminus P_i)}\mathbf{x}_{P_i}/x_l$ for all $1\leq j<i$. 
\end{Lemma}
\begin{proof} "$\Rightarrow$" Since $d_i=i-1$ we have that $\min(P_j\setminus P_i)\neq\min(P_k\setminus P_i)$ for all $1\leq j,k<i$ with $j\neq k$. Let $1\leq k<j<i$ and assume by contradiction that there exists $i_t,i_{t'}\in\supp(\mathbf{x}_{P_i})$, $i_t\neq i_{t'}$ such that
	\[\mathbf{x}_{P_j}=x_{\min(P_j\setminus P_i)}\frac{\mathbf{x}_{P_i}}{x_{i_t}},\ \mathbf{x}_{P_k}=x_{\min(P_k\setminus P_i)}\frac{\mathbf{x}_{P_i}}{x_{i_{t'}}}.
\]

By the proof of Lemma \ref{i-1}, we have that $\min(P_k\setminus P_i)=\min(P_k\setminus P_j)$. Since $j<i$, by Lemma \ref{i-1}, $d_j=j-1$ and there exists $j_{t''}\in\supp(\mathbf{x}_{P_j})$ such that $\mathbf{x}_{P_k}=x_{\min(P_k\setminus P_j)}\mathbf{x}_{P_j}/x_{j_{t''}}$. Replacing $\mathbf{x}_{P_k}$, we have that $\mathbf{x}_{P_j}x_{i_{t'}}=\mathbf{x}_{P_i}x_{j_{t''}}$. We note that $x_{i_{t'}}\neq x_{j_{t''}}$ since $\mathbf{x}_{P_j}\neq\mathbf{x}_{P_i}$. Hence, we have that
	\[x_{i_{t'}}x_{\min(P_j\setminus P_i)}\frac{\mathbf{x}_{P_i}}{x_{i_{t}}}=\mathbf{x}_{P_i}x_{j_{t''}},
\]
that is $x_{i_{t'}}x_{\min(P_j\setminus P_i)}=x_{j_{t''}}x_{i_t}.$ Since $x_{i_{t'}}\mid x_{j_{t''}}x_{i_t}$ and $x_{i_{t'}}\neq x_{j_{t''}}$, we have that $x_{i_{t'}}= x_{i_{t}}$ contradiction with our assumption.

"$\Leftarrow$" The statement is obvious.
\end{proof}
\section{A special class of subword complexes}

In this section we will consider only subword complexes $\Delta=\Delta(Q,\pi)$ such that the minimal monomial generating system of $I_{\Delta^{\vee}}$ has $r\leq n-\ell(\pi)+1$ elements, where $n$ is the size of $Q$, and for which $d_r=r-1$. In the following proposition, we construct classes of such subword complexes.

\begin{Proposition} Let $\pi\in W$ be an element and $\sigma_1\ldots\sigma_{\ell(\pi)}$ a reduced expression for $\pi$. Let $1\leq i\leq \ell(\pi)$ be a fixed integer and let $$Q=(\sigma_1,\sigma_2,\ldots,\sigma_{i-1},\sigma_i,\sigma_i,\ldots,\sigma_i,\sigma_{i+1},\ldots,\sigma_{\ell(\pi)})$$ be a word of size $n$ in $W$. Then the minimal monomial generating system of $I_{\Delta^{\vee}}$ has $N=n-\ell(\pi)+1$ elements and $d_N=N-1$.
\end{Proposition}
\begin{proof} Since $\sigma_1\ldots\sigma_{\ell(\pi)}$ is a reduced expression for $\pi$, any subword of $Q$ that represents $\pi$ is a copy of this reduced expression. So, $I_{\Delta^{\vee}}$ has the minimal monomial generating system $G(I_{\Delta^{\vee}})=\{x_1\ldots x_{i-1}x_jx_{n-{\ell(\pi)+i+1}}\ldots x_{n}|\ i\leq j\leq n-\ell(\pi)+i\}$. Hence, $|G(I_{\Delta^{\vee}})|=n-\ell(\pi)+1$. One may note that $d_N=N-1$, for $N= n-\ell(\pi)+1$.
\end{proof}

\begin{Lemma}\label{idelta}Let $\Delta$ be the subword complex $\Delta(Q,\pi)$ and suppose that the size of $Q$ is $n$. Assume that $G(I_{\Delta^{\vee}})=\{\mathbf{x}_{P_1},\ldots,\mathbf{x}_{P_r}\}$ with $\mathbf{x}_{P_1}>_{lex}\ldots>_{lex}\mathbf{x}_{P_r}$, $r\leq n-\ell(\pi)+1$ and $d_r=r-1$. Then there exists a unique $l\in\supp(\mathbf{x}_{P_r})$ such that
\[I_{\Delta^{\vee}}=\frac{\mathbf{x}_{P_r}}{x_l}(x_{\min(P_1\setminus P_r)},\ldots,x_{\min(P_{r-1}\setminus P_r)},x_l).
\]
\end{Lemma}
\begin{proof} Since $d_r=r-1$, the statement follows by Lemma \ref{unic l}. 
\end{proof}

\begin{Corollary}\label{ht} Let $\Delta$ be the subword complex $\Delta(Q,\pi)$ and suppose that the size of $Q$ is $n$. Assume that $G(I_{\Delta^{\vee}})=\{\mathbf{x}_{P_1},\ldots,\mathbf{x}_{P_r}\}$ with $\mathbf{x}_{P_1}>_{lex}\ldots>_{lex}\mathbf{x}_{P_r}$, $r\leq n-\ell(\pi)+1$ and $d_r=r-1$. Then $\height(I_{\Delta^{\vee}})=1$. 
\end{Corollary}

\begin{proof} By Lemma \ref{idelta}, there exists a unique $l\in\supp(\mathbf{x}_{P_r})$ such that
\[I_{\Delta^{\vee}}=\frac{\mathbf{x}_{P_r}}{x_l}(x_{\min(P_1\setminus P_r)},\ldots,x_{\min(P_{r-1}\setminus P_r)},x_l).
\]
One may note that for any $t\in\supp(\mathbf{x}_{P_r}/x_l)$ the ideal $(x_t)$ is a minimal prime ideal of $I_{\Delta^{\vee}}$. Hence $\height(I_{\Delta^{\vee}})=1$.
\end{proof}

Let $R=k[x_1,\ldots,x_n]$ be the polynomial ring over a field $k$.

\begin{Theorem}\label{Koszul} Let $\Delta$ be the subword complex $\Delta(Q,\pi)$ and let $n$ be the size of $Q$. Assume that $G(I_{\Delta^{\vee}})=\{\mathbf{x}_{P_1},\ldots,\mathbf{x}_{P_r}\}$ with $\mathbf{x}_{P_1}>_{lex}\ldots>_{lex}\mathbf{x}_{P_r}$, $r\leq n-\ell(\pi)+1$ and $d_r=r-1$. Then there exists a unique integer $l\in[n]$ such that the Koszul complex associated to the sequence $x_{\min(P_1\setminus P_r)},\ldots,x_{\min(P_{r-1}\setminus P_r)},x_l$ is isomorphic to the minimal graded free resolution of $k[\Delta^{\vee}]$.
\end{Theorem}

\begin{proof} By Lemma \ref{idelta} there exists a unique $l\in[n]$ such that 
	\[I_{\Delta^{\vee}}=\frac{\mathbf{x}_{P_r}}{x_l}(x_{\min(P_1\setminus P_r)},\ldots,x_{\min(P_{r-1}\setminus P_r)},x_l).
\] Hence, the multiplication by $\mathbf{x}_{P_r}/x_l$ defines an isomorphism of $R$--modules between $(x_{\min(P_1\setminus P_r)},\ldots,x_{\min(P_{r-1}\setminus P_r)},x_l)$ and $I_{\Delta^{\vee}}$. Since $x_{\min(P_1\setminus P_r)},\ldots,x_{\min(P_{r-1}\setminus P_r)},x_l$ is a regular sequence, the assertion follows.
\end{proof}

The following result is a simple consequence of Theorem \ref{Koszul}.
\begin{Corollary}\label{Betti}Let $\Delta$ be the subword complex $\Delta(Q,\pi)$ and suppose that the size of $Q$ is $n$. Assume that $G(I_{\Delta^{\vee}})=\{\mathbf{x}_{P_1},\ldots,\mathbf{x}_{P_r}\}$ with $\mathbf{x}_{P_1}>_{lex}\ldots>_{lex}\mathbf{x}_{P_r}$, $r\leq n-\ell(\pi)+1$ and $d_r=r-1$. Then
	\[\beta_i(I_{\Delta^\vee})=\left(\twoline{r}{i+1} \right), 
\]
for all $i$.
\end{Corollary}

\begin{Corollary}\label{hilbd} Let $\Delta$ be the subword complex $\Delta(Q,\pi)$ and let $n$ be the size of $Q$. Assume that $G(I_{\Delta^{\vee}})=\{\mathbf{x}_{P_1},\ldots,\mathbf{x}_{P_r}\}$ with $\mathbf{x}_{P_1}>_{lex}\ldots>_{lex}\mathbf{x}_{P_r}$, $r\leq n-\ell(\pi)+1$ and $d_r=r-1$. Then the Hilbert numerator of the Hilbert series is
	\[\mathcal{K}_{I_{\Delta^{\vee}}}(t)=\sum_{i=0}^{r-1}(-1)^i\left(\twoline{r}{i+1} \right)t^{i+\ell(\pi)}.
\]
\end{Corollary}
\begin{proof} Since $I_{\Delta^{\vee}}$ has a $\ell(\pi)$--linear resolution, $\beta_{ij}(I_{\Delta^{\vee}})=0$ for all $j\neq i+\ell(\pi)$. Hence $\beta_i(I_{\Delta^{\vee}})=\beta_{i,i+\ell(\pi)}(I_{\Delta^{\vee}})$. Since $\projdim(I_{\Delta^{\vee}})=r-1$,
\[\mathcal{K}_{I_{\Delta^{\vee}}}(t)=\sum_{i=0}^{r-1}(-1)^i\beta_i(I_{\Delta^{\vee}})t^{i+\ell(\pi)}.
\]
Hence, using Corollary \ref{Betti}, we get that the Hilbert numerator is 
\[\mathcal{K}_{I_{\Delta^{\vee}}}(t)=\sum_{i=0}^{r-1}(-1)^i\left(\twoline{r}{i+1} \right)t^{i+\ell(\pi)}.
\]
\end{proof}

\begin{Corollary}\label{numw} Let $Q$ be a word in $W$ of size $n$ that contains $\pi$, $\Delta$ the subword complex  $\Delta(Q,\pi)$ and $G(I_{\Delta^{\vee}})=\{\mathbf{x}_{P_1},\ldots,\mathbf{x}_{P_r}\}$ with $\mathbf{x}_{P_1}>_{lex}\ldots>_{lex}\mathbf{x}_{P_r}$, $r\leq n-\ell(\pi)+1$ and $d_r=r-1$. Then there are $\left(\twoline{r}{j+1} \right)$ subwords $P$ of $Q$ such that $\delta(P)=\pi$ and $|P|=j+\ell(\pi)$ for $0\leq j\leq r-1$.
\end{Corollary}
\begin{proof} By \cite[Lemma 4.2]{KnMi}, in the fine grading, the Hilbert numerator of $I_{\Delta^\vee}$ is
\begin{eqnarray*}
  \mathcal{K}_{I_{\Delta^\vee}}( t_1,\ldots,t_n) &=& \sum_\twoline{P \subseteq Q}
  {\delta(P)=\pi} (-1)^{|P|-\ell(\pi)} \mathbf{t}^P.
\end{eqnarray*}
where $\mathbf{t}^P=\prod\limits_{\sigma_i\in P}t_i$.
In $\mathbb{Z}$--grading, we have 
\[
  \mathcal{K}_{I_{\Delta^\vee}}(t) = \sum_\twoline{P \subseteq Q}
  {\delta(P)=\pi} (-1)^{|P|-\ell(\pi)} t^{|P|}=\sum_\twoline{P \subseteq Q}
  {\delta(P)=\pi}\sum^n_\twoline{j=\ell(\pi)}
{|P|=j} (-1)^{j-\ell(\pi)} t^j=\]
  \[=\sum_\twoline{P \subseteq Q}
  {\delta(P)=\pi}\sum^{n-\ell(\pi)}_\twoline{j=0}
  {|P|=j+\ell(\pi)} (-1)^{j} t^{j+\ell(\pi)}=\sum_{j=0}^{n-\ell(\pi)}(-1)^j m_{j+\ell(\pi)} t^{j+\ell(\pi)}\]
where we denoted by $m_{j+\ell(\pi)}$ the number of the subwords $P$ of $Q$ such that $\delta(P)=\pi$ and $|P|=j+\ell(\pi)$. Comparing with the formula from Corollary \ref{hilbd}, we obtain that $m_{j+\ell(\pi)}=\left(\twoline{r}{j+1} \right)$ for all $0\leq j\leq r-1$ and $m_{j+\ell(\pi)}=0$ for all $r\leq j\leq n-\ell(\pi)$.
\end{proof}

\begin{Corollary} Let $Q$ be a word in $W$ of size $n$ that contains $\pi$, $\Delta$ the subword complex $\Delta(Q,\pi)$ and $G(I_{\Delta^{\vee}})=\{\mathbf{x}_{P_1},\ldots,\mathbf{x}_{P_r}\}$ with $\mathbf{x}_{P_1}>_{lex}\ldots>_{lex}\mathbf{x}_{P_r}$, $r\leq n-\ell(\pi)+1$ and $d_r=r-1$. Then $\Delta$ is a simplicial sphere if and only if $r=n-\ell(\pi)+1$.
\end{Corollary}
\begin{proof} By \cite[Corollary 3.8]{KnMi}, $\Delta$ is a simplicial sphere if $\delta(Q)=\pi$. Hence, in the Hilbert numerator, the coefficient of $t^{|Q|}$ must be non-zero. By Corollary \ref{numw}, the coefficient of $t^n$ is $m_n=\left(\twoline{r}{n-\ell(\pi)+1}\right)$. Hence $m_n\neq 0$ if and only if $r=n-\ell(\pi)+1$.
\end{proof}

\begin{Proposition}\label{complint} Let $\Delta$ be the subword complex $\Delta(Q,\pi)$ and let $n$ be the size of $Q$. Assume that $G(I_{\Delta^{\vee}})=\{\mathbf{x}_{P_1},\ldots,\mathbf{x}_{P_r}\}$ with $\mathbf{x}_{P_1}>_{lex}\ldots>_{lex}\mathbf{x}_{P_r}$, $r \leq n-\ell(\pi)+1$ and $d_r=r-1$. Then $k[\Delta]$ is a complete intersection ring.
\end{Proposition}
\begin{proof} By Lemma \ref{idelta}, we have that there exists a unique integer $l$ such that \[I_{\Delta^{\vee}}=\frac{\mathbf{x}_{P_r}}{x_l}(x_{\min(P_1\setminus P_r)},\ldots,x_{\min(P_{r-1}\setminus P_r)},x_l).
\]
Hence
\[I_{\Delta^{\vee}}=\left(\bigcap_{k\in\supp\left(\mathbf{x}_{P_r}/x_l\right)}(x_k)\right)\cap\left(x_{\min(P_1\setminus P_r)},\ldots, x_{\min(P_{r-1}\setminus P_r)},x_l\right)
\]
and
	\[I_{\Delta}=\left(x_{\min(P_1\setminus P_r)}\ldots x_{\min(P_{r-1}\setminus P_r)}x_l\right)+\left(x_k:k\in\supp\left(\mathbf{x}_{P_r}/x_l\right)\right).
\]
Since $\supp(\mathbf{x}_{P_r}/x_l)\cap\left\{\min(P_1\setminus P_r),\ldots, \min(P_{r-1}\setminus P_r),l\right\}=\emptyset$, $I_{\Delta}$ is a complete intersection ideal. 
\end{proof}

\begin{Corollary} Let $\Delta$ be the subword complex $\Delta(Q,\pi)$ and suppose that the size of $Q$ is $n$. Assume that $G(I_{\Delta^{\vee}})=\{\mathbf{x}_{P_1},\ldots,\mathbf{x}_{P_r}\}$ with $\mathbf{x}_{P_1}>_{lex}\ldots>_{lex}\mathbf{x}_{P_r}$, $r\leq n-\ell(\pi)+1$ and $d_r=r-1$. Then there exists a unique integer $l$ such that the Koszul complex associated to the sequence $x_{\min(P_1\setminus P_r)}\ldots x_{\min(P_{r-1}\setminus P_r)}x_l,\ x_i\ :i\in\supp(\mathbf{x}_{P_r}/x_l)$ is the minimal graded free resolution of $k[{\Delta}]$.
\end{Corollary}
\begin{proof} By Proposition \ref{complint}, we have that $G(I_{\Delta})=\{x_{\min(P_1\setminus P_r)}\ldots x_{\min(P_{r-1}\setminus P_r)}x_l,$ $x_i\ :i\in\supp(\mathbf{x}_{P_r}/x_l)\}$. The statement follows.
\end{proof}

\begin{Proposition} Let $\Delta$ be the subword complex $\Delta(Q,\pi)$, and let $n$ be the size of $Q$. Assume that $G(I_{\Delta^{\vee}})=\{\mathbf{x}_{P_1},\ldots,\mathbf{x}_{P_r}\}$ with $\mathbf{x}_{P_1}>_{lex}\ldots>_{lex}\mathbf{x}_{P_r}$, $r\leq n-\ell(\pi)+1$ and $d_r=r-1$. Then $k[\Delta^{\vee}]$ is Cohen--Macaulay if and only if $I_{\Delta^{\vee}}$ is a principal monomial ideal.
\end{Proposition}
\begin{proof} By Eagon--Reiner theorem \cite{EaRe}, $k[\Delta^{\vee}]$ is Cohen--Macaulay if and only if $I_{\Delta}$ has a linear resolution. In particular, $I_{\Delta}$ is generated in one degree. The statement follows by Proposition \ref{complint}. 
\end{proof}

\end{document}